\documentclass[11pt]{article}

\RequirePackage[OT1]{fontenc}
\RequirePackage{amsthm,amsmath}
\RequirePackage[numbers]{natbib}
\RequirePackage[colorlinks,citecolor=blue,urlcolor=blue]{hyperref}

\usepackage{graphicx}
\usepackage{color}
\usepackage{amsmath}
\usepackage{amssymb}
\usepackage{amscd}
\usepackage{framed}
\usepackage{bbm}
\usepackage{natbib}
\usepackage{hyperref}
\usepackage{comment}

\usepackage{fancyhdr,a4wide}
\usepackage{amsthm}

\newcommand{\R}{\mathbb{R}}
\newcommand{\inr}[1]{\langle #1 \rangle}

\newcommand{\E}{\mathbb{E}}

\newcommand{\eps}{\varepsilon}

\newcommand{\vertiii}[1]{{\left\vert\kern-0.25ex\left\vert\kern-0.25ex\left\vert #1
    \right\vert\kern-0.25ex\right\vert\kern-0.25ex\right\vert}}

\newtheorem{Theorem}{Theorem}[section]
\newtheorem{Lemma}[Theorem]{Lemma}

\newtheorem{Definition}[Theorem]{Definition}

\newtheorem{Remark}[Theorem]{Remark}
\newtheorem{Example}[Theorem]{Example}

\bibpunct{[}{]}{,}{n}{,}{,}

\def\IND{\mathbbm{1}}

\newcommand{\var}{\mathrm{Var}}

\def\IND{\mathbbm{1}}
\begin{document}

\title{Robust covariance estimation under $L_4-L_2$ norm equivalence.}
\author{Shahar Mendelson\thanks{Mathematical Sciences Institute, The Australian National University and LPSM, Sorbonne Universit{\'e}, \href{mailto:shahar.mendelson@upmc.fr}{shahar.mendelson@upmc.fr}} \and Nikita Zhivotovskiy\thanks{Higher School of Economics. Now visiting researcher at Google. Nikita Zhivotovskiy was supported by RSF grant No. 18-11-00132 \href{mailto:nikita.zhivotovskiy@phystech.edu}{nikita.zhivotovskiy@phystech.edu}}}
\date{}

\maketitle
\begin{abstract}
Let $X$ be a centered random vector taking values in $\R^d$ and let $\Sigma= \E (X\otimes X)$ be its covariance matrix. We show that if $X$ satisfies an $L_4-L_2$ norm equivalence (sometimes referred to as the bounded kurtosis assumption), there is a covariance estimator $\hat{\Sigma}$ that exhibits the optimal performance one would expect had $X$ been a gaussian vector.
The procedure also improves the current state-of-the-art regarding high probability bounds in the subgaussian case (sharp results were only known in expectation or with constant probability).

In both scenarios the new bounds do not depend explicitly on the dimension $d$, but rather on the effective rank of the covariance matrix $\Sigma$.
\end{abstract}

\section{Introduction}

The question of estimating the covariance of a random vector has been studied extensively in recent years (see, e.g., \cite{Koltch17, Lounici14, Minsk18a, Minsk17a, Minsk18} and references therein). To formulate the problem, let $X$ be a zero mean random vector taking its values in $\mathbb{R}^d$ and denote the covariance matrix by $\Sigma = \E (X\otimes X)$. Given a sample $X_1,...,X_N$ consisting of independent random vectors that are distributed according to $X$, the goal is to
select a matrix $\hat{\Sigma}$ that approximates $\Sigma$. While there are various notions of approximation, the focus of this note is on approximation with respect to the ($\ell_2 \to \ell_2$) operator norm, which from here on is denoted by $\| \ \|$.

One way of viewing the question of covariance estimation (with respect to any norm), is as a vector mean estimation problem. Indeed, if one sets $W=X \otimes X$, then $\E W = \Sigma$, and since one is given a sample $X_1,...,X_N$, the vectors $(X_i \otimes X_i)_{i=1}^N$ are $N$ independent copies of $W$. Thus, a matrix $\hat{W}$ that is a good approximation of the mean $\E W$ with respect to the underlying norm is a solution to the problem of estimating the covariance of $X$ with respect to that norm.

An immediate outcome of this simple observation is that the empirical mean
$$
\hat{\Sigma} = \frac{1}{N}\sum_{i=1}^N W_i = \frac{1}{N}\sum_{i=1}^N X_i \otimes X_i,
$$
which is the trivial choice for estimating the true mean, is a poor estimator unless the random vector $W$ has a `nice' tail behaviour (see, for example, the discussion in \cite{Mendelson18}). An example of a positive result of that flavour is Theorem 9 in \cite{Koltch17}, and to formulate it we need some definitions.
\begin{Definition} \label{def:effective-rank}
The \emph{effective rank} of a positive semidefinite square matrix $A \in \mathbb{R}^{d \times d}$ is given by
\begin{equation} \label{eq:eff-rank}
\mathbf{r}(A) = \frac{{\rm Tr}(A)}{\|A\|}.
\end{equation}
\end{Definition}
Clearly, $\mathbf{r}(A) \le d$ but the gap between $\mathbf{r}(A)$ and $d$ may be substantial. Recall that the $\psi_2$-norm of a centred real-valued random variable $Y$ is defined by
$$
\|Y\|_{\psi_2} = \inf\{c > 0: \E\exp(Y^2/c^2) \le 2\},
$$
and that there are absolute constants $c$ and $C$ such that
$$
c \|Y\|_{\psi_2} \leq \sup_{p \geq 2} \frac{\|Y\|_{L_p}}{\sqrt{p}} \leq C \|Y\|_{\psi_2}.
$$

\begin{Definition} \label{def:psi-2}
A random vector $X$ with values in $\R^d$ and with the mean $\mu$ is $L$-subgaussian if for every $t \in \R^d$ and every $p \geq 2$,
\begin{equation}
\label{eq:momentcomparison}
(\E|\inr{X-\mu,t}|^p)^{\frac{1}{p}} \leq L \sqrt{p}(\E\inr{X-\mu,t}^2)^{\frac{1}{2}}.
\end{equation}
\end{Definition}

It is standard to verify that a centred random vector is $L$-subgaussian if and only if its one-dimensional marginals $X_t=\inr{X,t}$ satisfy that
$\|X_t\|_{\psi_2} \leq c L \|X_t\|_{L_2}$ where $c$ is an absolute constant.

Among the class of $L$-subgaussian random vectors are vectors whose distribution is multivariate normal (denoted by $\mathcal{N}(\mu, \Sigma)$) and in which case $L$ is an absolute constant.  Another simple example are vectors $X$ whose components are independent copies of a zero mean random variable $Y$ that satisfies $\|Y\|_{\psi_2} < \infty$. Indeed, it is standard to show that for such a random vector and any $t \in \R^d$
\[
\sup\limits_{p \ge 2}\frac{(\E|\inr{X,t}|^p)^{\frac{1}{p}}}{\sqrt{p}} \le c\left(\sum\limits_{i = 1}^dt_i^2\|Y_i\|^2_{\psi_2}\right)^{\frac{1}{2}} = L(\E|\inr{X,t}|^2)^{1/2},
\]
where $c$ is an absolute constant and $L = c\|Y\|_{\psi_2}/\|Y\|_{L_2}$. 

\begin{Remark}
Observe that a different notion of subgaussian random vectors sometimes appears in literature: that a centred vector $X$ is called subgaussian if
\begin{equation}
\label{subgaussc}
\sup\limits_{t \in S^{d - 1}}\sup\limits_{p \ge 2} \frac{(\E|\inr{X, t}|^p)^{\frac{1}{p}}}{\sqrt{p}}= C < \infty
\end{equation}
where $S^{d-1}$ is the Euclidean unit sphere in $\R^d$. In other words, according to this notion, a centred random vector is subgaussian if all its one-dimensional marginals have a finite $\psi_2$ norm, and those norms are all bounded by $C$. Unlike the notion in Definition \ref{def:psi-2}, this does not imply a $\psi_2-L_2$ norm equivalence of one-dimensional marginals of the random vector. As a result, the constant $C$ in \eqref{subgaussc} may change dramatically under linear transformations of $X$, while the factor $L$ in \eqref{eq:momentcomparison} does not.
\end{Remark}

Throughout this note the notion of a subgaussian random vector that is used is the one from Definition \ref{def:psi-2}.

With all the required definitions in place, one may formulate the covariance estimate from \cite{Koltch17}.

\begin{Theorem} \label{thm:Kolt}
For every $L \geq 1$ there exists a constant $c(L)$ for which the following holds.
Let $X$ be an $L$-subgaussian random vector. Then with probability at least $1-\delta$
\begin{equation} \label{eq:subgauss}
\left\|\frac{1}{N}\sum_{i=1}^N  X_i \otimes X_i - \Sigma\right\| \leq c(L) \|\Sigma\|\left(\sqrt{\frac{\mathbf{r}(\Sigma)}{N}} + \frac{\mathbf{r}(\Sigma)}{N} + \sqrt{\frac{\log ({2}/{\delta})}{N}} + \frac{\log ({2}/{\delta})}{N}\right)
\end{equation}
\end{Theorem}
It was also shown in \cite{Koltch17} that if $G$ is a zero mean gaussian vector (and in particular it satisfies the conditions of Theorem \ref{thm:Kolt}) with covariance $\Sigma$  then
\[
\E\left\|\frac{1}{N}\sum_{i=1}^N  G_i \otimes G_i - \Sigma\right\| \gtrsim \|\Sigma\|\max\left\{\sqrt{\frac{\mathbf{r}(\Sigma)}{N}}, \frac{\mathbf{r}(\Sigma)}{N}\right\}.
\]
Hence, there is no room for improvement in the deviation estimate of the empirical mean from the true one at the constant confidence level. Of course, that does not imply that the empirical mean is an optimal covariance estimator --- even for a gaussian vector, let alone for a general subgaussian random vector. In fact, as we explain in what follows, there are far better covariance estimators than \eqref{eq:subgauss} when the confidence parameter $\delta$ is small.

\vskip0.3cm
Just as in the one-dimensional mean-estimation problem, once the problem is more `heavy-tailed' the performance of the empirical mean deteriorates quickly and a different procedure has to be used. And that is also the case for covariance estimation. The current state-of-the-art for covariance estimation in heavy-tailed situation is \cite{Minsk18} (see Corollary 4.1 there and similar results in \cite{Minsk18a, Minsk17a}), in which $X$ is assumed to satisfy an $L_4-L_2$ norm equivalence.
\begin{Definition} \label{def:L-4-L-2}
A random vector $X$ with mean $\mu$ satisfies the $L_4-L_2$ norm equivalence with a constant $L \ge 1$ if for every $t \in \R^d$,
$$
(\E \inr{X-\mu,t}^4)^{\frac{1}{4}} \leq L (\E \inr{X-\mu,t}^2)^{\frac{1}{2}}.
$$
\end{Definition}

Note that if $X$ is $L$-subgaussian then it satisfies an $L_4-L_2$ norm equivalence with constant $2L$. At the same time, for an $L_4-L_2$ equivalence the linear forms $\inr{X,t}$ need not have higher moments that the fourth one; in particular, $X$ need not be $L$-subgaussian. Another formulation of the same condition is that for any direction, the \emph{kurtosis}\footnote{The kurtosis of the random variable $Y$ is equal to $\frac{\E(Y - \E Y)^4}{(\E(Y - \E Y)^2)^2}$.} of the corresponding one-dimensional marginal is bounded by $L$.

\begin{Remark}
In Appendix \ref{app:equiv}  one can find two examples that demonstrate the difference between a random vector $X$ being $L$-subgaussian (which implies an $\psi_2-L_2$ norm equivalence of the centred marginals of $X$) and $X$ satisfying an $L_4-L_2$ norm equivalence.
\end{Remark}

The current state of the art estimate for random vectors that satisfy Definition \ref{def:L-4-L-2} is as follows:
\begin{Theorem}[\cite{Minsk18}] \label{thm:L-4-L-2}
For every $L \geq 1$ there are constants $c(L)$ and $c^\prime(L)$ that depend only on $L$ and for which the following holds. Let $X$ satisfy an $L_4-L_2$ norm equivalence with constant $L$.
For $0<\delta<1$ there is an estimator $\tilde{\Sigma}_\delta$ that satisfies
\begin{equation} \label{eq:L-4-L-2}
\|\tilde{\Sigma}_\delta - \Sigma\| \leq c(L) \|\Sigma\|\sqrt{\frac{\mathbf{r}(\Sigma)}{N} \cdot \left(\log d + \log(1/\delta)\right) }
\end{equation}
with probability at least $1-\delta$, provided that $N \geq c^{\prime}(L) \mathbf{r}(\Sigma)(\log d + \log(1/\delta))$.
\end{Theorem}

\begin{Remark}
Let us mention that the procedure from \cite{Minsk18} requires prior information on the values of $\|\Sigma\|$ and ${\mathbf{r}(\Sigma)}$ up to some absolute multiplicative constant---an assumption we shall return to in what follows. In fact, a significant part of our analysis is devoted to obtaining estimates on these parameters, and our approach is an alternative to Lepski's method used in \cite{Minsk17a, Minsk18}.
\end{Remark}

Observe that if $\delta$ is smaller than $1/d$, the error guaranteed by Theorem \ref{thm:L-4-L-2} is of the order of
\begin{equation} \label{eq:error-L-4-L-2}
\|\Sigma\| \sqrt{\frac{\mathbf{r}(\Sigma)}{N}} \sqrt{\log(1/\delta)},
\end{equation}
which turns out to be far from optimal as we now explain. 

\vskip0.3cm

To put \eqref{eq:error-L-4-L-2} in some perspective, let us examine possible benchmarks for general mean estimation problems and see how those compare with \eqref{eq:subgauss}, \eqref{eq:L-4-L-2} and \eqref{eq:error-L-4-L-2} when applied to covariance estimation.

\subsection{Optimality in mean estimation}
Let $W$ be a random vector with mean $\mu$ and set $\vertiii{ \ }$ to be an arbitrary norm. Let  $B^{\circ}$ be the unit ball of the dual norm to $\vertiii{ \ }$, and denote by $\hat{\mu}$ a mean-estimator constructed using an independent sample $W_1,...,W_N$. As it happens, a lower bound on the performance of $\hat{\mu}$ is
\begin{equation} \label{eq:lower-weak}
\frac{R}{\sqrt{N}} \sqrt{\log(1/\delta)}
\end{equation}
where
\begin{equation} \label{eq:R}
R=\sup_{x^* \in B^\circ} \left(\E (x^*(W-\mu))^2\right)^{\frac{1}{2}}.
\end{equation}
Indeed, for every $x^* \in B^{\circ}$
$$
\vertiii{\hat{\mu}-\mu} \geq |x^*(\hat{\mu}-\mu)|=|x^*(\hat{\mu})-x^*(\mu)|;
$$
therefore, if there is a procedure for which $\vertiii{\hat{\mu}-\mu} \leq \eps$ with probability $1-\delta$, then on the same event the procedure automatically performs with accuracy $\eps$ and confidence $1-\delta$ for each one of the real-valued mean-estimation problems associated with the random variables $x^*(W)$, $x^* \in B^\circ$. By a lower bound (Proposition 6.1 from \cite{Catoni12}) on real-valued mean estimation problems when $W$ is a gaussian vector, the best possible mean-estimation error for each $x^*(W)$ is
$$
\sqrt{\frac{{\rm var}\bigl(x^*(W)\bigr)}{N}}\sqrt{\log(1/\delta)},
$$
and taking the `worst' $x^* \in B^\circ$ leads to \eqref{eq:lower-weak}.
\vskip0.3cm
Although \eqref{eq:lower-weak} is part of the story, it is unlikely it is the whole story. Intuitively, \eqref{eq:lower-weak} takes into account the effect of one-dimensional marginals of $W$ rather than the entire geometry of the distribution. It stands to reason that an additional `global' parameter is called for---one that reflects the entire structure of $W$ and the geometry of the norm. Moreover, that parameter should reflect the difficulty of the estimation problem at the constant confidence level.

To give an example of such a result, a (sharp) lower bound from \cite{Catoni12} on the mean estimation problem when $W$ is a gaussian random vector is the following: if $\vertiii{\hat{\mu}-\mu} \leq \eps$ with probability at least $1-\delta$ then
\begin{equation} \label{eq:lower-gaussian}
\eps \geq \frac{c}{\sqrt{N}} \left(\E\vertiii{W-\mu} + R \sqrt{\log(1/\delta)}\right);
\end{equation}
hence, the `global parameter' in the gaussian case is just the mean $\E\vertiii{W-\mu}$.

Let us examine \eqref{eq:lower-gaussian} more carefully, in the hope that it would lead us towards the right answer for general random vectors. Note that by setting $\delta=\exp(-p)$, the gaussian random variable $W$ satisfies that
$$
\sqrt{\log(1/\delta)} (\E (x^*(W-\mu))^2)^{\frac{1}{2}} \sim \sqrt{p}(\E (x^*(W-\mu))^2)^{\frac{1}{2}} \sim (\E |x^*(W-\mu)|^p)^{\frac{1}{p}}.
$$
At the same time, the \emph{strong-weak norm inequality\footnote{By `strong norm' we mean the $L_1$ norm of $\vertiii{W-\mu}$, while the `weak norm' is just the largest $L_p$ norm of a marginal $x^*(W-\mu)$ for $x^* \in B^{\circ}$.}} for gaussian vectors (see, e.g., \cite{Latala08}) implies that
\begin{align*}
\left( \E \vertiii{\frac{1}{N}\sum_{i=1}^N W_i - \mu}^p \right)^{\frac{1}{p}} \leq & \E \vertiii{\frac{1}{N}\sum_{i=1}^N W_i - \mu} + c  \sup_{x^* \in B^{\circ}} \Bigl(\E \bigl| x^*\bigl(\frac{1}{N}\sum_{i=1}^N W_i - \mu\bigr)\bigr|^p \Bigr)^{\frac{1}{p}}
\\
= & \frac{1}{\sqrt{N}} \Bigl(\E \vertiii{W - \mu} + c \sup_{x^* \in B^{\circ}} \bigl(\E | x^*(W - \mu)|^p \bigr)^{\frac{1}{p}}\Bigr),
\\
= & \frac{1}{\sqrt{N}} \Bigl(\E \vertiii{W - \mu} + c^\prime \sqrt{p}\sup_{x^* \in B^{\circ}} \bigl(\E | x^*(W - \mu)|^2 \bigr)^{\frac{1}{2}}\Bigr),
\end{align*}
where $c$ and $c^\prime$ are absolute constants. Thus, the lower bound of \eqref{eq:lower-gaussian} implies that the best possible performance of a mean estimator of a gaussian vector matches a strong-weak norm inequality.  To see that these bounds are of the same order, one needs to use Markov's inequality and optimize with respect to $p$, where the right choice is indeed $p \sim \log(1/\delta)$.

This leads to a natural conjecture: that the best possible performance in a general mean estimation problem is given by a gaussian-like strong-weak norm inequality, and that there is a procedure that performs with that accuracy/confidence tradeoff.

Recently, a general mean estimation procedure was introduced in \cite{Mendelson18} that exhibits this type of a ``strong-weak" behaviour. To formulate the result, let $W$ be an arbitrary random vector taking values in $\R^d$ and with mean $\mu$, let $G$ be the zero mean gaussian random vector with the same covariance as $W$ and set
$$
Y_N =   \frac{1}{N}\sum_{i=1}^N (W_i - \mu),
$$
where $W_1,...,W_N$ are independent copies of $W$. Let $\vertiii{ \ }$ be a norm, set $B^\circ$ to be the unit ball of the dual norm, and put
$$
R=\sup_{x^* \in B^\circ} \left(\E (x^*(W-\mu))^2\right)^{\frac{1}{2}}.
$$

\begin{Theorem} \label{thm:LM-mod} \cite{Mendelson18}
For $0<\delta <1 $ there is a procedure $\tilde{\mu}_\delta$ such that
$$
\vertiii{\tilde{\mu}_\delta - \mu} \leq c \max \left\{\E\vertiii{Y_N}, \ \ \frac{\E \vertiii{G}}{\sqrt{N}} + \frac{R}{\sqrt{N}}\sqrt{\log(1/\delta)}\right\}.
$$

The mean estimation procedure is defined as follows: let $T={\rm ext}(B^{\circ})$ to be the set of extreme points in $B^{\circ}$.
\begin{framed}
\begin{itemize}
\item For the wanted confidence parameter  $0<\delta<1$, let $n=\log(1/\delta)$ and set $m=N/n$.
\item Let $(I_j)_{j=1}^n$ be the natural partition of $\{1,...,N\}$ to blocks of cardinality $m$ and given a sample $W_1,...,W_N$ set $Z_j=\frac{1}{m}\sum_{i \in I_j} W_j$.
\item For $x^* \in T$ and $\eps>0$, set
$$
S_{x^*}(\eps) = \left\{y \in \R^d \ : \ \left|x^*(Y) - x^*(Z_j) \right| \leq \eps \ \text{for more than} \ n/2 \ \text{blocks} \right\},
$$
and define
$$
S(\eps) = \bigcap_{x^* \in T} S_{x^*}(\eps).
$$

\item Set $\eps_0 = \inf \{\eps>0: S(\eps) \not=\emptyset\}$, and let $\tilde{\mu}_\delta$ be any vector in $\bigcap_{\eps>\eps_0} S(\eps)$.
\end{itemize}
\end{framed}
\end{Theorem}

The main result of this note (which is formulated in the next section), is that the right application of Theorem \ref{thm:LM-mod} leads to an (almost) optimal covariance estimator: the procedure performs as if $X$ were a gaussian vector even if $X$ only satisfies an $L_4-L_2$ norm equivalence, and the accuracy/confidence tradeoff obeys the strong-weak inequality one would expect.

\subsection{From mean estimation to covariance estimation}

In what follows, we assume without loss of generality that $X$ is symmetric and zero mean. We may do so because if $X^\prime$ is an independent copy of $X$ then $Z=(X-X^\prime)/\sqrt{2}$ is symmetric and has the same covariance as $X$. It also satisfies an $L_4-L_2$ norm equivalence if $X$ does. Thus, given a random sample $X_1, \ldots, X_N$ sampled independently according to $X$ one may consider the sample
\[
\frac{1}{\sqrt{2}}(X_1 - X_2), \ldots, \frac{1}{\sqrt{2}}(X_{N - 1} - X_N),
\]
consisting of $N/2$ independent copies of $Z$, and perform the procedure with respect to that sample. 

The natural choice of a random vector in Theorem \ref{thm:LM-mod} is $W = X \otimes X$, but as it happens, a better alternative is to use a truncated version of $X$ instead of the original one:
\begin{Definition} \label{def:truncation}
Let
$$
\beta = \left(\frac{{\rm Tr}(\Sigma)\|\Sigma\|N}{\gamma}\right)^{\frac{1}{4}},
$$
and let
$$
\tilde{X} = X \IND_{\{\|X\|_2 \leq \beta\}}.
$$
In the $L$-subgaussian case set $\gamma = 1$  and when $X$ only satisfies $L_4-L_2$ norm equivalence, let $\gamma=\log {\mathbf{r}}(\Sigma)$. Also denote $\tilde{\Sigma} = \E(\tilde{X} \otimes \tilde{X})$.
\end{Definition}

\begin{Definition} \label{weakvar}
Given the random vector $X$ taking its values in $\mathbb{R}^d$ define
\begin{equation} \label{rsquared}
R^2_{X} = \sup_{u,v \in S^{d-1}} \E \left(v^T(X \otimes X - \E X \otimes X)u\right)^2,
\end{equation}
\end{Definition}

The quantity $R^2_X$ is sometimes referred to as the \emph{weak variance} of a random matrix. 

\vskip0.3cm
As was mentioned previously, the main result of this note is the existence of an estimator whose performance improves both \eqref{eq:subgauss} and \eqref{eq:L-4-L-2} and is an optimal (or very close to being optimal) covariance estimation procedure.

\vskip0.3cm
The estimator is constructed in three stages: the first stage leads to a data-dependent estimate on ${\rm Tr}(\Sigma)$; the second stage is based on the estimated value of  ${\rm Tr}(\Sigma)$ established in the first stage and its outcome is a data-dependent estimate on the value of $\|\Sigma\|$; the last stage receives as input the results of two first stages and the third part of the sample and returns the wanted estimator of $\Sigma$. A key point in the analysis of this procedure is that one only needs to estimate ${\rm Tr}(\Sigma)$ and $\|\Sigma\|$ up to absolute multiplicative constant factors and that simplifies the problem considerably.

The performance of the procedure is summarized in this, our main result.
\begin{Theorem} \label{thm:main}
Let $X$ be a zero mean random vector with (an unknown) covariance matrix $\Sigma$ and let $\| \ \|$ be its operator norm. Using the notation of Definition \ref{def:truncation} and Definition \ref{weakvar}, for any $0<\delta<1$, there is a procedure that receives as data the sample $X_1,...,X_N$, returns a matrix $\hat{\Sigma}_\delta$ and satisfies:
\begin{description}
\item{(1)} If $X$ is $L$-subgaussian and $N \geq c^{\prime}(L)(\mathbf{r}(\Sigma) + \log(1/\delta))$, then with probability at least $1-\delta$,
$$
\|\hat{\Sigma}_\delta-\Sigma\| \leq c(L) \left(\|\Sigma\| \sqrt{\frac{\mathbf{r}(\Sigma)}{N}} + \frac{R_{\tilde{X}}}{\sqrt{N}} \sqrt{\log(1/\delta)} \right);
$$
\item{(2)} If $X$ satisfies an $L_4-L_2$ norm equivalence and $N \geq c^\prime(L)(\mathbf{r}(\Sigma) \log \mathbf{r}(\Sigma) {+ \log(1/\delta)})$
then with probability at least $1-\delta$,
\begin{equation}
\label{seceq}
\|\hat{\Sigma}_\delta-\Sigma\| \leq c(L) \left( \|\Sigma\|\sqrt{\frac{\mathbf{r}(\Sigma) \log (\mathbf{r}(\Sigma))}{N}}+\frac{R_{\tilde{X}}}{\sqrt{N}} \sqrt{\log(1/\delta)} \right).
\end{equation}
\end{description}
In both cases  $R_{\tilde{X}} \leq c(L)\|\Sigma\|$ and $c(L), c^\prime(L)$ are constants that depend only on $L$.
\end{Theorem}

\begin{Remark}
Note that the estimates in Theorem \ref{thm:main} do not depend on the dimension $d$; instead, they depend only on $ \mathbf{r}(\Sigma)$ which may be small even if $d$ tends to infinity. This is important in view of the recent results on covariance estimation in Banach spaces \cite{Koltch17}.
\end{Remark}

The estimate in Theorem \ref{thm:main} is actually a strong-weak norm inequality---as if $X$ were gaussian (up to the logarithmic term in \eqref{seceq}). Indeed, let $G$ be the zero mean gaussian random vector that has the same covariance as $X$
and set $N \geq \mathbf{r}(\Sigma)$. As noted previously,
$$
\|\Sigma\|\sqrt{\frac{\mathbf{r}(\Sigma)}{N}}  \sim  \E \left\|\frac{1}{N}\sum_{i=1}^N G_i \otimes G_i - \Sigma \right\|,
$$
with the left-hand side being the `strong term' from Theorem \ref{thm:main}. Moreover, the term involving $R_X$ is actually the natural weak term associated with the operator norm. Indeed, recall the well-known fact that the dual norm to the operator norm is the nuclear norm. And, since a linear functional $z$ acts on the matrix $x$ via trace duality---that is $z(x)=[z,x]:={\rm Tr}(z^Tx)$ --- it follows, for example, from \cite{so99} that the extreme points of the dual unit ball $B^\circ$ are
$$
\left\{ u \otimes v : u,v \in S^{d-1} \right\}.
$$
Thus,
\begin{equation*}
R^2_{\tilde{X}} = \sup_{x^* \in B^\circ}\E\bigl(x^*( \tilde{X}\otimes  \tilde{X} - \tilde{\Sigma})\bigr)^2 = \sup\limits_{u, v \in S^{d - 1}}\E \bigl(v^T ( \tilde{X} \otimes  \tilde{X} - \tilde{\Sigma})u\bigr)^2,
\end{equation*}
and in particular, by \eqref{eq:lower-weak} the weak term $(R_{\tilde{X}}/\sqrt{N}) \sqrt{\log(1/\delta)}$ appearing in Theorem \ref{thm:main} is sharp.

\vskip0.3cm

As a result, and up to the logarithmic factor in $(2)$, Theorem \ref{thm:main} implies that the estimator $\hat{\Sigma}_\delta$ performs as if $X$ were gaussian, even though it can be very far from gaussian.

\vskip0.3cm

Let us compare the outcome of Theorem \ref{thm:main} to the current state of the art that was mentioned previously.
In the subgaussian setup Theorem \ref{thm:main} improves Theorem \ref{thm:Kolt} because there are situations in which $R_{\tilde X}$ is significantly smaller than $\|\Sigma\|$ (see such an example in what follows). And, under an $L_4-L_2$ norm equivalence scenario the improvement is more dramatic: on top of an improvement in the logarithmic factor appearing in the `strong' term, the `weak' term, $(R_{\tilde{X}}/\sqrt{N}) \sqrt{\log(1/\delta)}$ is significantly smaller than the corresponding estimate of $\|\Sigma\| \sqrt{{\mathbf{r}(\Sigma)}/N} \sqrt{\log(1/\delta)}$ from Theorem \ref{thm:L-4-L-2}.

\vskip0.3cm

The proof of Theorem \ref{thm:main} is presented in the following section.

\vskip0.3cm
We end this introduction with some notation. Throughout, absolute constants are denoted by $c, c_1, \ldots, c^{\prime}, \ldots$ and their value may change from line to line. Constants that depend on a parameter $L$ are denoted by $c(L)$, $a \lesssim b$ means that there is an absolute constant $c$ such that $a \leq cb$, and $a \sim b$ means that $c b \leq a \leq c_1b$. When the constants depend on $L$ we write $a \lesssim_L b$ and $a \sim_L b$ respectively. 

\section{Proof of Theorem \ref{thm:main}}
Consider the truncated vector $\tilde{X}$ introduced in Definition \ref{def:truncation} but for now for an arbitrary level of truncation. Therefore, let $\alpha \ge 0$ and with a minor abuse of notation, redefine 
\begin{equation}
\label{deftilde}
\tilde{X} = X \IND_{\{\|X\|_2 \leq \alpha\}}\quad \text{and} \quad \tilde{\Sigma} = \E\tilde{X}\otimes\tilde{X},
\end{equation}
First, note that by the symmetry of $X$, $\tilde{X}$ is symmetric as well. Second, for every $p \geq 2$ and any $u \in \R^d$,
$$
\|\inr{\tilde{X},u}\|_{L_p} = (\E|\inr{\tilde{X},u}|^p)^{\frac{1}{p}} \leq (\E|\inr{X,u}|^p)^{\frac{1}{p}}.
$$
Hence, if $X$ is $L$-subgaussian then $\|\inr{\tilde{X},u}\|_{L_p} \leq L \sqrt{p} \|\inr{X,u}\|_{L_2}$, and if $X$ satisfies $L_4-L_2$ norm equivalence with constant $L$ then $\|\inr{\tilde{X},u}\|_{L_4} \leq L \|\inr{X,u}\|_{L_2}$.

\vskip0.3cm

More important features of $\tilde{X}$ have to do with its covariance matrix $\tilde{\Sigma}$ and trace ${\rm Tr}(\Sigma)$:
\begin{Lemma} \label{lemma:tilde-sigma}
If $X$ is zero mean and satisfies an $L_4-L_2$ norm equivalence with constant $L$, then
\begin{equation} \label{eq:tilde-sigma-1}
\|\tilde{\Sigma} - \Sigma\| \leq c(L)\frac{ \|\Sigma\|\rm{Tr}(\Sigma)}{\alpha^2},
\end{equation}
and
\begin{equation} \label{eq:tilde-sigma-2}
\bigl|{\rm Tr}(\tilde{\Sigma}) - {\rm Tr}(\Sigma) \bigr| \leq c(L)\frac{{\rm Tr}^2(\Sigma)}{\alpha^2},
\end{equation}
where $c(L)$ is a constant that depends only on $L$.
\end{Lemma}

\proof Observe that
\begin{align*}
 \|\tilde{\Sigma}- \Sigma\| &= \sup_{u,v \in S^{d-1}} \left| u^T \left( \E (X \otimes X) - \E (\tilde{X} \otimes \tilde{X}) \right) v \right|
\\
 &=  \sup_{u,v \in S^{d-1}} \left| \E \inr{X,u} \inr{X,v} \IND_{\{\|X\|_2 > \alpha\}}\right|
\\
&\leq  \sup_{u,v \in S^{d-1}} \left(\E\inr{X,u}^4\right)^{\frac{1}{4}} \cdot \left(\E\inr{X,v}^4\right)^{\frac{1}{4}} \cdot Pr^{\frac{1}{2}} (\|X\|_2 \geq \alpha).
\end{align*}
By the $L_4-L_2$ norm equivalence,
$$
\sup_{u \in S^{d-1}} \left(\E\inr{X,u}^4\right)^{\frac{1}{4}} \leq L \sup_{u \in S^{d-1}} \left(\E\inr{X,u}^2\right)^{\frac{1}{2}} = L\|\Sigma\|^{\frac{1}{2}}
$$
and
\begin{align}
\E \|X\|_2^4 = & \E \left(\sum_{i=1}^{d} \inr{X,e_i}^2 \right)^2 \leq \E \sum_{i,j} \inr{X,e_i}^2 \inr{X,e_j}^2  \leq \sum_{i,j} \bigl(\E \inr{X,e_i}^4\bigr)^{\frac{1}{2}} \bigl(\E \inr{X,e_j}^4\bigr)^{\frac{1}{2}} \nonumber
\\
\label{traceeq}
\leq & L^2 \sum_{i,j} \E\inr{X,e_i}^2 \cdot \E \inr{X,e_j}^2 = L^2 \sum_{i,j} \Sigma_{ii} \Sigma_{jj} = L^2 \bigl({\rm Tr}(\Sigma)\bigr)^2.
\end{align}
Clearly,
\begin{equation} \label{eq:tail}
Pr^{\frac{1}{2}} (\|X\|_2 \geq \alpha) \leq  \left(\frac{\E \|X\|_2^4}{\alpha^4}\right)^{\frac{1}{2}} \leq L\frac{ \bigl({\rm Tr}(\Sigma)\bigr)}{\alpha^2}
\end{equation}
and combining the two observations,
\begin{equation} \label{eq:sigma-approximation}
\|\tilde{\Sigma}-\Sigma\| \leq c(L)\frac{ \|\Sigma\|\rm{Tr}(\Sigma)}{\alpha^2},
\end{equation}
as claimed. Turning to the second part of the lemma, note that
$$
{\rm Tr}(\Sigma) = \sum_{i=1}^d \E \inr{X,e_i}^2 \ \ \ {\rm and} \ \ \ {\rm Tr}(\tilde{\Sigma}) = \sum_{i=1}^d \E \inr{X,e_i}^2\IND_{\{\|X\|_2 \leq \alpha\}}.
$$
Therefore, by the $L_4-L_2$ norm equivalence and \eqref{eq:tail},
\begin{align*}
\bigl|{\rm Tr}(\tilde{\Sigma})-{\rm Tr}(\Sigma)\bigr| = & \sum_{i=1}^d \E \inr{X,e_i}^2 \IND_{\{\|X\|_2 > \alpha\}} \leq \sum_{i=1}^d \E \bigl(\inr{X,e_i}^4\bigr)^{\frac{1}{2}} Pr^{\frac{1}{2}}(\|X\|_2 > \alpha)
\\
\leq & L^2 \left(\sum_{i=1}^d \E \inr{X,e_i}^2 \right) Pr^{\frac{1}{2}}(\|X\|_2 > \alpha) \leq c(L)\frac{{\rm Tr}^2(\Sigma)}{\alpha^2}.
\end{align*}
\endproof

The core component in the estimation procedure is denoted by $\hat{\Sigma}_{\delta, \alpha}$, and its definition for a truncation parameter $\alpha>0$ is as follows:
\begin{framed}
{\bf The estimator $\hat{\Sigma}_{\delta, \alpha}$}
\vskip+10pt
Let $\alpha > 0,\ 0<\delta<1$  and consider the given sample $X_1,...,X_N$. Set $\tilde{X}_i = X_i \IND_{\{\|X_i\|_2 \leq \alpha\}}$.
\begin{itemize}
\item Let $n = \log(1/\delta)$ and split the sample to $n$ blocks $I_j$, each one of cardinality $m = N/n$; set $M_j = \frac{1}{m}\sum_{i \in I_j} \tilde{X}_i \otimes \tilde{X}_i$.
\item Let $T = \{(u,v):\ u, v \in S^{d - 1}\}$ and for $\eps>0$ and a pair $(u,v)$ let
$$
S_{u,v}(\eps) = \left\{Y \in \mathbb{R}^{d \times d}: \left|v^T\left(M_j - Y\right)u\right| \leq \eps \ \text{for more than} \ n/2 \ \text{blocks} \right\}.
$$
\item Set
$$
S(\eps) = \bigcap_{(u,v) \in T} S_{u,v}(\eps).
$$
\item Let $\eps_0 = \inf \{\eps>0: S(\eps) \not=\emptyset\}$ and choose $\hat{\Sigma}_{\delta, \alpha}$ to be any matrix that satisfies
\begin{equation} \label{truncest}
\hat{\Sigma}_{\delta, \alpha} \in \bigcap_{\eps>\eps_0} S(\eps).
\end{equation}
\end{itemize}
\end{framed}

While the right truncation level is given in Definition \ref{def:truncation}, namely 
$$
\beta=\left(\frac{{\rm Tr}(\Sigma)\|\Sigma\|N}{\gamma}\right)^{\frac{1}{4}},
$$
its definition depends on the identities of ${\rm Tr}(\Sigma)$ and $\|\Sigma\|$, which are unknown. To address this issue one first invokes a median-of-means estimator, denoted by $\hat{\varphi}_1$, and show that with high probability, 
$$
\frac{1}{2} {\rm Tr}(\Sigma) \leq \hat{\varphi}_1 \leq 2{\rm Tr}(\Sigma).
$$
Then $\hat{\Sigma}_{\delta,\alpha}$ is performed on an independent part of the sample and at a truncation level of $\alpha \sim \hat{\varphi}_1$, i.e., of the order of ${\rm Tr}(\Sigma)$. The outcome in an estimator $\hat{\varphi}_2$ that satisfies 
$$
\frac{\|\Sigma\|}{2} \leq \hat{\varphi}_2 \leq 2 \|\Sigma\|
$$
with high probability. 

The combination of $\hat{\varphi}_1$ and $\hat{\varphi}_2$ allows one to identify $\beta$ up to an absolute constant. With that information, $\hat{\Sigma}_{\delta,\alpha}$ is preformed again, this time at the `correct level', resulting in a matrix that is a fine approximation of $\Sigma$. 

\vskip0.3cm

With that in mind, the core of the proof of Theorem \ref{thm:main} is the next Lemma.
\begin{Lemma} \label{truncestlemma}
Using the notation introduced previously, the following holds for $\hat{\Sigma}_{\delta, \alpha}$:
\begin{description}
\item{(1)} If $X$ is $L$-subgaussian, then with probability at least $1-\delta$,
$$
\|\hat{\Sigma}_{\delta, \alpha}-\tilde{\Sigma}\| \leq c(L) \left(\|\Sigma\| \left(\sqrt{\frac{\mathbf{r}(\Sigma)}{N}} + \frac{\mathbf{r}(\Sigma)}{N}\right) + \frac{R_{\tilde X}}{\sqrt{N}} \sqrt{\log(1/\delta)} \right).
$$
\item{(2)} If $X$ satisfies an $L_4-L_2$ norm equivalence, $N \geq c^\prime(L)\mathbf{r}(\Sigma) \log \mathbf{r}(\Sigma)$ and
$$
c_1(L)\sqrt{{\rm Tr}(\Sigma)} \le \alpha \le c_2(L)\left(\frac{{\rm Tr}(\Sigma)\|\Sigma\|N}{\log \mathbf{r}(\Sigma)}\right)^{\frac{1}{4}}
$$
then with probability at least $1-\delta$,
\[
\|\hat{\Sigma}_{\delta, \alpha}-\tilde{\Sigma}\| \leq c(L)\left(\|\Sigma\|\sqrt{\frac{\mathbf{r}(\Sigma)\log \mathbf{r}(\Sigma)}{N}} + \frac{R_{\tilde{X}}}{\sqrt{N}} \sqrt{\log(1/\delta)}\right),
\]
\end{description}
where $R_{\tilde X}$ is as in \eqref{rsquared}. 

In both cases  $R_{\tilde{X}} \leq c(L)\|\Sigma\|$ and $c(L), c^{\prime}(L), c_1(L), c_2(L)$ are constants that depend only on $L$.
\end{Lemma}

The proof of the lemma is presented in Section \ref{truncestlemmaproof}. Assuming its validity let us complete the proof of Theorem \ref{thm:main}. From this point on and without the loss of generality, assume that the given sample is of cardinality $3N$, as that only affects the constant factors appearing in the bounds.

\subsubsection*{{\bf Stage 1. Estimation of $\rm{Tr}(\Sigma)$}}
The first goal is to use the first $N$ observations $X_1, \ldots, X_N$ to construct the estimator $\hat{\varphi}_1$, for which, with high probability  $\hat{\varphi}_1 \sim \rm{Tr}(\Sigma)$.
Since
\[
{\rm Tr}(\Sigma) = \E\sum_{i=1}^d \inr{X,e_i}^2,
\]
a standard median-of-means estimator $\hat{\varphi}_1$ of $\E\sum_{i=1}^d \inr{X,e_i}^2$ (see \cite{Nemir83} for what is by now a standard argument) satisfies that with probability at least $1 - \delta$,
\[
|\hat{\varphi}_1 - {\rm Tr}(\Sigma)| \le c\sqrt{\var\left(\sum_{i=1}^d \inr{X,e_i}^2\right)\frac{\log (1/\delta)}{N}}.
\]
Using \eqref{traceeq},
$$
\var\left(\sum_{i=1}^d \inr{X,e_i}^2\right) \le (L^2 - 1){\rm Tr}(\Sigma)^2,
$$
and therefore,
\[
|\hat{\varphi}_1 - {\rm Tr}(\Sigma)| \le c(L){\rm Tr}(\Sigma)\sqrt{\frac{\log (1/\delta)}{N}}.
\]
Hence, if $N \ge c'(L)\log (1/\delta)$, then with probability at least $1-\delta$ one has 
\begin{equation} \label{eq:phi-1}
\frac{1}{2}{\rm Tr}(\Sigma) \leq \hat{\varphi}_1 \leq 2 {\rm Tr}(\Sigma).
\end{equation}

\subsubsection*{{\bf Stage 2. Estimation of $\|\Sigma\|$}}
In this stage, the second part of the sample $X_{N + 1}, \ldots, X_{2N}$ is utilized, and the procedure receives as an additional input $\hat{\varphi}_1$ that satisfies \eqref{eq:phi-1}. To ease notation, one may assume that ${\rm Tr}(\Sigma)$ is known and set $\alpha = \kappa(L)\sqrt{{\rm Tr}(\Sigma)}$, where $\kappa(L)$ is a constant that depends only on $L$.

 Using the notation from \eqref{deftilde} and by Lemma \ref{lemma:tilde-sigma} it follows that
\[
 \|\tilde{\Sigma} - \Sigma\| \le  c(L)\frac{\|\Sigma\|}{\kappa^2(L)},
\]
and 
\[
\bigl|{\rm Tr}(\tilde{\Sigma})-{\rm Tr}(\Sigma)\bigr| \le c(L)\frac{{\rm Tr}(\Sigma)}{\kappa^2(L)}.
\]

In the $L$-subgaussian case, invoking Lemma \ref{truncestlemma} and the triangle inequality,
\begin{align*}
\|\hat{\Sigma}_{\delta, \alpha} - \Sigma\| \le & \|\tilde{\Sigma} - \Sigma\| + \|\hat{\Sigma}_{\delta, \alpha} - \tilde{\Sigma}\| \le \frac{\|\Sigma\|}{10} + c(L) \|\Sigma\|\left( \sqrt{\frac{\mathbf{r}(\Sigma)}{N}} + \frac{\mathbf{r}(\Sigma)}{N}+ \sqrt{\frac{\log(1/\delta)}{N}} \right)
\\
\leq & \frac{\|\Sigma\|}{2},
\end{align*}
provided that $N \ge c'(L)(\mathbf{r}(\Sigma) + \log (1/\delta))$ for a large enough constant $c'(L)$ and $c(L)/\kappa^2(L) \le \frac{1}{10}$. In that case, setting $\hat{\varphi}_2 = \|\hat{\Sigma}_{\delta, \alpha}\|$, it follows that 
\begin{equation} \label{varphisecond}
\frac{\|\Sigma\|}{2} \leq \hat{\varphi}_2 \leq 2\|\Sigma\|.
\end{equation}
Finally, in the case of $L_4-L_2$ norm equivalence, and again by Lemma \ref{truncestlemma}, one has that \eqref{varphisecond} holds as long as $N \ge c'(L)(\mathbf{r}(\Sigma)\log \mathbf{r}(\Sigma) + \log (1/\delta))$. Indeed, since $R_{\tilde{X}} \le c(L)\|\Sigma\|$, one has
\[
\|\hat{\Sigma}_{\delta, \alpha} - \Sigma\| \le \frac{1}{10}\|\Sigma\| + c(L) \|\Sigma\|\left( \sqrt{\frac{\mathbf{r}(\Sigma) \log(\mathbf{r}(\Sigma))}{N}} + \sqrt{\frac{\log(1/\delta)}{N}} \right) \leq \frac{\|\Sigma\|}{2},
\]
as required.

\subsubsection*{{\bf Stage 3. Estimation of $\Sigma$}}
The final step uses the third part of the sample $X_{2N + 1}, \ldots, X_{3N}$. Consider a truncation level $\beta$ as in Definition \ref{def:truncation}, and which, thanks to the first two stages, can be estimated by $\hat{\beta}$ up to an absolute multiplicative factor. Therefore, to ease notation again, simplicity, assume that $\beta$ itself is knows.

For that choice of truncation parameter consider $\tilde{X}$ and $\tilde{\Sigma}$ as in \eqref{deftilde} and let $\hat{\Sigma}_\delta = \hat{\Sigma}_{\delta, \beta}$.

By the triangle inequality,
$$
\|\hat{\Sigma}_\delta - \Sigma\| \leq \|\hat{\Sigma}_\delta - \tilde{\Sigma}\| + \|\tilde{\Sigma}-\Sigma\|,
$$
and by Lemma \ref{lemma:tilde-sigma} the quantity $\|\tilde{\Sigma}-\Sigma\|$ is smaller than the wanted accuracy for the chosen level $\beta$. The required bound on $\|\hat{\Sigma}_\delta - \Sigma\|$ follows immediately from Lemma \ref{truncestlemma}, and Theorem \ref{thm:main} follows by taking the union bound over the events analyzed in three stages and combining the conditions on $N$.

\section{Proof of Lemma \ref{truncestlemma}}

Thanks to Theorem \ref{thm:LM-mod}, the proof of Lemma \ref{truncestlemma} follows once one establishes sufficient control on $\E \|Y_N\|$, $\E\|G\|$ and $R_{\tilde X}$.

\label{truncestlemmaproof}
\subsubsection*{{\bf Controlling $R_{\tilde X}$}}
The required estimate on $R_{\tilde X}$ for an arbitrary truncation level $\alpha$ is presented in the next Lemma.
\begin{Lemma} \label{rlemma}
Assume that $X$ is zero mean and satisfies an $L_4-L_2$ norm equivalence with constant $L$. Setting $\mathbf{v}^2(X) = \sup_{v \in S^{d-1}} \E\inr{X,v}^4$ one has that 
\[
R_{\tilde{X}} \le \mathbf{v}(X) \lesssim_L  \|\Sigma\|.
\]
\end{Lemma}

\begin{proof}
For every $u,v \in S^{d-1}$, $\E \inr{\tilde{X},v} \inr{\tilde{X},u} = v^T \tilde{\Sigma} u$; therefore,
\begin{align*}
\E \bigl(v^T (\tilde{X} \otimes \tilde{X} - \tilde{\Sigma})u\bigr)^2 = & \E \inr{\tilde{X},v}^2 \inr{\tilde{X},u}^2 - (v^T\tilde{\Sigma} u)^2 \leq \E \inr{\tilde{X},v}^2 \inr{\tilde{X},u}^2
\\
\leq & \bigl(\E \inr{\tilde{X},v}^4 \bigr)^{\frac{1}{2}} \cdot \bigl(\E \inr{\tilde{X},u}^4 \bigr)^{\frac{1}{2}},
\end{align*}
implying that $R_{\tilde{X}} \le \mathbf{v}(X)$.

Also, recalling that $X$ satisfies and $L_4-L_2$ norm equivalence,
$$
\E\inr{X,v}^4 \leq L^4 \bigl(\E \inr{X,v}^2\bigr)^2 \leq L^4 \|\Sigma\|^2
$$
implying that $\mathbf{v}(X) \leq L^2 \|\Sigma\|$, as claimed.

\end{proof}

\subsubsection*{{\bf Controlling $\E\|G\|$ and $\E \|Y_N\|$}}
In the context of Theorem \ref{thm:LM-mod}, $G$ is the zero mean gaussian vector on $\R^{d \times d}$ whose covariance coincides with that of $W=\tilde{X} \otimes \tilde{X}$. Instead of dealing with that vector directly, note that
\begin{equation} \label{eq:gaussian-control}
\E\|G\| \leq \liminf_{N \to \infty} \sqrt{N} \E\|Y_N\|.
\end{equation}
Indeed, for every finite set $T^\prime$, 
$$
\E \|G\| = \sup_{T^\prime \subset B^{\circ}, \ T^\prime \ {\rm is \ finite} } \E \max_{x^* \in T^\prime} x^*(G),
$$
and by the multivariate CLT, for every finite set $T^\prime$, $\left\{ N^{-1/2}\sum_{i=1}^N x^*(W_i-\E W) : x^* \in T^\prime \right\}$ converges weakly to $\left\{ x^*(G) : x^* \in T^\prime\right\}$. Hence, \eqref{eq:gaussian-control} follows from tail integration.

\vskip0.3cm
Thanks to \eqref{eq:gaussian-control}, all that remains is to bound $\E\|Y_N\|$.

\subsubsection*{The subgaussian case}
Fix an integer $N$ and note that
\begin{equation} \label{eq:quadratic}
\left\|\frac{1}{N} \sum_{i=1}^N \tilde{X}_i \otimes \tilde{X}_i - \tilde{\Sigma} \right\| = \sup_{u \in S^{d-1}} \left|\frac{1}{N} \sum_{i=1}^N \inr{\tilde{X}_i,u}^2 - \E \inr{\tilde{X}_i,u}^2 \right|,
\end{equation}
which is the supremum of a quadratic empirical process indexed by $S^{d-1}$. Such empirical processes have been studied extensively (see, e.g., \cite{Mendelson10, Mendelson16, Mendelson07}), mainly using chaining methods. As it happens, quadratic \emph{subgaussian} processes may be controlled in terms of a natural metric invariant of the indexing class---the so-called \emph{$\gamma_2$ functional}\footnote{Rather than defining the $\gamma_2$ functional, we refer the reader to \cite{Talagrand14} for a detailed exposition on the topic, and to \cite{Mendelson10, Mendelson16, Mendelson07} for the study of the quadratic empirical process in this and more general situations.}. In the case of \eqref{eq:quadratic}, the indexing class is $S^{d-1}$ whose elements are viewed as linear functionals on $\R^d$, and the underlying metric is the $\psi_2$ norm endowed by the random vector $\tilde{X}$. By Corollary 1.9 from \cite{Mendelson07} it follows that
\begin{equation} \label{eq:MPT}
\E \sup_{u \in S^{d-1}} \left|\frac{1}{N} \sum_{i=1}^N \inr{\tilde{X}_i,u}^2 - \E \inr{\tilde{X}_i,u}^2 \right| \leq c \left({\cal D} \frac{\gamma_2(S^{d-1},\psi_2(\tilde{X}))}{\sqrt{N}} + \frac{\gamma_2^2(S^{d-1},\psi_2(\tilde{X}))}{N}\right),
\end{equation}
where $c$ is an absolute constant and
$$
{\cal D} = {\cal D}(S^{d-1},\psi_2) = \sup_{u \in S^{d-1}} \|\inr{\tilde{X},u}\|_{\psi_2} \sim \sup_{u \in S^{d-1}} \sup_{p \geq 2} \frac{\left(\E|\inr{\tilde{X},u}|^p\right)^{\frac{1}{p}}}{\sqrt{p}}.
$$

To estimate \eqref{eq:MPT} one requires two facts (see, e.g., \cite{Talagrand14} for more details). Firstly, a general property of the $\gamma_2$ functional is monotonicity in $d$: if $(T,d)$ is a metric space and $d^\prime$ is another metric on $T$ which satisfies that for every $t_1,t_2 \in T$, $d(t_1,t_2) \leq \kappa d^\prime(t_1,t_2)$, then
$$
\gamma_2(T,d) \leq \kappa \gamma_2(T,d^\prime).
$$
Here, for every $p \geq 2$ and $u \in \R^d$,
$$
\bigl(\E|\inr{\tilde{X},u}|^p\bigr)^{\frac{1}{p}} \leq \bigl(\E|\inr{X,u}|^p\bigr)^{\frac{1}{p}} \leq L \sqrt{p}\bigl(\E|\inr{X,u}|^2\bigr)^{\frac{1}{2}},
$$
implying that
$$
\|\inr{\tilde{X},u}\|_{\psi_2} \leq L \|\inr{X,u}\|_{L_2};
$$
hence, $\gamma_2(S^{d-1},\psi_2(\tilde{X})) \leq L \gamma_2(S^{d-1},L_2(X))$.

Secondly, by Talagrand's majorizing measures theorem, if $G$ is a zero mean gaussian random vector with the same covariance as $X$ then
$$
\gamma_2(S^{d-1},L_2(X)) \leq c \E \sup_{u \in S^{d-1}} \inr{G,u} \leq c \bigl(\E \|G\|_2^2\bigr)^{\frac{1}{2}} = c \sqrt{{\rm Tr}(\Sigma)},
$$
for a some absolute constant $c$.

Finally, again thanks to the fact that $X$ is $L$-subgaussian,
$$
{\cal D} \leq L \sup_{u \in S^{d-1}} \|\inr{X,u}\|_{L_2} = L \|\Sigma\|^{\frac{1}{2}}.
$$
Therefore, by \eqref{eq:MPT}, for every $N$,
$$
\E\|Y_N\| \leq c(L) \left( \|\Sigma\|^{1/2} \sqrt{\frac{{\rm Tr}(\Sigma)}{N}} + \frac{{\rm Tr}(\Sigma)}{N} \right),
$$
and in particular, $\liminf_{N \to \infty} \sqrt{N} \E\|Y_N\| \leq c(L) \|\Sigma\|^{1/2} \sqrt{{\rm Tr}(\Sigma)}$.

This completes the proof of the first part of Lemma \ref{truncestlemma}.
\endproof

\subsubsection*{{\bf $L_4-L_2$ norm equivalence}}
Just as in the subgaussian case, the key issue is finding a suitable estimate on $\E\|Y_N\|$. Thanks to the fact that $\tilde{X}$ is a truncated random vector, one may apply a version of the matrix Bernstein inequality.

We invoke Corollary 7.3.2 from the survey \cite{Tropp15} (which is a slightly modified version of the original result from \cite{Minsk17}): if $Z$ is a random vector which satisfies that $\|Z \otimes Z\| \leq \beta $ almost surely, and $B=\E(Z \otimes Z)^2$, then
\begin{equation} \label{eq:matrix-Bernstein}
\E \left\|\frac{1}{N}\sum_{i=1}^N Z_i \otimes Z_i - \E (Z \otimes Z) \right\| \leq c \left(\sqrt{\frac{\|B\| \log(\mathbf{r}(B))}{N}} +  \frac{\beta\log(\mathbf{r}(B))}{N} \right).
\end{equation}
Here, $Z=X\IND_{\{\|X\| \leq \alpha\}}$ for $\alpha$ as in Definition \ref{def:truncation}, and all that remains is to estimate $\|B\|$ and $\mathbf{r}(B)$.

It is straightforward to verify that
$$
c \|\tilde \Sigma\| {\rm Tr}(\tilde{\Sigma}) \leq \|B\| \leq c_1(L) \|\Sigma\| {\rm Tr}({\Sigma}) \ \ \ {\rm and} \ \ \ {\rm Tr}(B) \leq c_1(L) \bigl({\rm Tr}({\Sigma})\bigr)^2:
$$
the upper estimates on $\|B\|$ and ${\rm Tr}(B)$  follow from a direct computation and the fact that $X$ satisfies an $L_4-L_2$ norm equivalence (see, e.g., Lemma 4.1 in \cite{Minsk18}); the lower estimate is an outcome of the FKG inequality (see Corollary 5.1 in the supplementary material to \cite{Minsk17a}).

Turning to the upper bound $\mathbf{r}(B)$, by Lemma \ref{lemma:tilde-sigma} and using its notation, both $\|\tilde{\Sigma}\|$ and ${\rm Tr}(\tilde{\Sigma})$ are equivalent up to multiplicative constant factors to $\|\Sigma\|$ and ${\rm Tr}(\Sigma)$ respectively, as long as $\alpha \ge c_2(L) \sqrt{{\rm Tr}(\Sigma)}$; hence,
$\mathbf{r}(B) \lesssim_L \mathbf{r}(\Sigma)$.

Finally, observe that $\|Z \otimes Z\|  =  \|Z\|_2^2 \leq \alpha^2$. By \eqref{eq:matrix-Bernstein} and the fact that $N \gtrsim_L \mathbf{r}(\Sigma) \log \mathbf{r}(\Sigma)$,
\begin{align} \label{eq:Y-N-in-proof}
\E \|Y_N\| &\leq  c(L) \left(\|\Sigma\|^{1/2} \sqrt{\frac{{\rm Tr}(\Sigma)\log \mathbf{r}(\Sigma)}{N}} + \alpha^2 \frac{\log \mathbf{r}(\Sigma)}{N}\right) \nonumber
\\
&= c(L)\|\Sigma\|\left(\sqrt{\frac{\mathbf{r}(\Sigma)\log \mathbf{r}(\Sigma)}{N}}  +  \frac{\alpha^2\mathbf{r}(\Sigma)\log \mathbf{r}(\Sigma)}{{\rm Tr}(\Sigma) N}\right).
\end{align}
In particular,
$$
\liminf_{N \to \infty} \sqrt{N} \E \|Y_N\| \leq c^{\prime}(L)\|\Sigma\|\sqrt{\mathbf{r}(\Sigma) \log \mathbf{r}(\Sigma)},
$$
provided that $\alpha = \alpha(N)$ satisfies
\[
\liminf_{N \to \infty}\frac{\alpha^2\mathbf{r}(\Sigma)\log \mathbf{r}(\Sigma)}{{\rm Tr}(\Sigma) \sqrt{N}} \le c^{\prime\prime}(L)\sqrt{\mathbf{r}(\Sigma) \log \mathbf{r}(\Sigma)}.
\]
That is the case if 
\begin{equation} \label{alphacond}
\alpha \le c_2(L)\left(\frac{{\rm Tr}(\Sigma)\|\Sigma\|N}{\log \mathbf{r}(\Sigma)}\right)^{\frac{1}{4}}.
\end{equation}
Finally, observe that when \eqref{alphacond} holds,  
$$
\frac{\alpha^2\mathbf{r}(\Sigma)\log \mathbf{r}(\Sigma)}{{\rm Tr}(\Sigma) N} \le c_2^2(L)\sqrt{\frac{\mathbf{r}(\Sigma)\log \mathbf{r}(\Sigma)}{N}},
$$
and combined with \eqref{eq:Y-N-in-proof} this completes the proof of second part of the lemma. 
\endproof

\subsubsection*{{\bf Concluding remarks}}
We start this section with an alternative way of estimating $\|\Sigma\|$ that does not require the knowledge of either ${\rm Tr}(\Sigma)$ or $L$, and does not have the extra factor $\log \mathbf{r}(\Sigma)$ appearing in the condition on $N$ when $X$ satisfies an $L_4-L_2$ norm equivalence. The drawback of this approach is that the bound depends on the dimension $d$, rather than on $\mathbf{r}(\Sigma)$.

\vskip0.3cm
\emph{Sketch of the argument.} Let $\mathcal{N}$ be a minimal $1/4$ cover of $S^{d-1}$ with respect to the Euclidean norm. Thus, $\|\Sigma\| \sim \sup_{u \in \mathcal{N}}u^T\Sigma u$. For any fixed $u$, the median of means estimator $\hat{\varphi}_{2, u}$ of $\E u^T X\otimes X u$ satisfies that with probability at least $1 - \delta$,
\[
|\hat{\varphi}_{2, u} - u^T\Sigma u| \le c(L)\|\Sigma\|\sqrt{\frac{\log (1/\delta)}{N}},
\]
because $\var\left(u^TX\otimes X u\right) \le L^4 \|\Sigma\|^2$. Finally, recalling that $|\mathcal{N}| \le 9^d$, the union bound shows that with probability at least $1 - \delta$
\[
\sup\limits_{u \in \mathcal{N}}|\hat{\varphi}_{2, u} - u^T\Sigma u| \le c_1(L)\|\Sigma\|\sqrt{\frac{d + \log (1/\delta)}{N}}.
\]
Therefore, when $N \ge c_1'(L)(d + \log (1/\delta))$, one has that $\sup\limits_{u \in \mathcal{N}} \hat{\varphi}_{2, u}  \sim \|\Sigma\|$ with probability at least $1-\delta$.

\endproof

\vskip0.4cm
We end this note with an example showing that there could be a substantial gap between $R_X$ and $\|\Sigma\|$ (and in a similar way, between $R_X$ and $\mathbf{v}(X)$), which is a reason for the sub-optimality of Theorem \ref{thm:Kolt} (Theorem $9$ in \cite{Koltch17}).

\begin{Example}
\label{rexample}
Let $(\eps_i)_{i=1}^d$ be independent, symmetric, $\{-1,1\}$-valued random variables, and set $\alpha_{1} > \ldots > \alpha_{d} \ge 0$. Let $X^{(i)} =\alpha_i \eps_i$ and consider $X = (X^{(1)}, \ldots, X^{(d)})$. Since the $X^{(i)}$'s are centered, independent and subgaussian with a constant subgaussian parameter, then $X$ is a centered, $L$-subgaussian random vector for some absolute constant $L$.

Let $\Sigma = \E (X \otimes X)$ and note that $\|\Sigma\| = \alpha_1^2$, $\mathbf{r}(\Sigma) = \sum\limits_{i = 1}^d \alpha_i^2/\alpha_1^2$ and
\begin{align*}
&\E\bigl(v^T (X \otimes X - \Sigma)u\bigr)^2 = \E\bigl(\sum_{i \neq j}v_iu_jX^{(i)}X^{(j)} \bigr)^2 = \sum_{i \neq j}\alpha_i^2 \alpha_j^2 (v_i^2u_j^2  + v_iv_ju_iu_j)
\\
&\le (\alpha_1\alpha_2)^2\bigl(\sum_{i, j}(v_iu_j)^2 + |v_iv_ju_iu_j|\bigr) \le (\alpha_1\alpha_2)^2\left(\|v\|^2\|u\|^2 + \inr{|v|, |u|}^2\right) \le 2(\alpha_1\alpha_2)^2.
\end{align*}
Hence, 
\begin{equation}
\label{rineq}
R_X \le \sqrt{2}\alpha_1\alpha_2 \le \alpha_1^2 = \|\Sigma\|,
\end{equation}
and the gap between $R_X$ and $\|\Sigma\|$ may be arbitrary large.

Inequality \eqref{rineq} is the best one can hope for in general. Indeed, let $Y$ be a centered random vector taking its values in $\mathbb{R}^d$, set $\Sigma = \E(Y \otimes Y)$ and consider $R_Y$. It follows that
\begin{align*}
\|\E(Y \otimes Y - \Sigma)^2\| &= \left\|\E(Y \otimes Y - \Sigma)\sum_{i = 1}^d e_ie_i^T(Y \otimes Y - \Sigma)\right\|
\\
&\le \sum_{i = 1}^d\sup\limits_{v \in S^{d - 1}}\E\left(e_i^T(Y \otimes Y - \Sigma)v\right)^2 \le dR^2_{Y}.
\end{align*}
As before, Corollary 5.1 in \cite{Minsk17a} implies that $\|\E(Y \otimes Y)^2\| \ge {\rm Tr}(\Sigma)\|\Sigma\|$. Therefore,
\begin{equation*}
dR_{Y}^2 \ge \|\E(Y \otimes Y - \Sigma)^2\| \ge \|\E(Y \otimes Y)^2\| - \|\Sigma^2\| \ge ({\rm Tr}(\Sigma))\|\Sigma\| - \|\Sigma\|^2
\end{equation*}
and
\begin{equation}
\label{rsecineq}
R_{Y} \ge \sqrt{\frac{\mathbf{r}(\Sigma) - 1}{d}}\|\Sigma\|,
\end{equation}
which is optimal when $\mathbf{r}(\Sigma) \sim d$. 
\end{Example}

\bibliography{covbib}

\newpage
\appendix
\section{Subgaussian vs. norm equivalence} \label{app:equiv}
The first example we present is the class of $L$-\emph{subexponential} random vectors. These vectors satisfy
\[
(\E|\inr{X-\mu,t}|^p)^{\frac{1}{p}} \leq L p(\E\inr{X-\mu,t}^2)^{\frac{1}{2}}
\]
for every $p \ge 2$; in particular, $X$ satisfies an $L_4-L_2$ norm equivalence with constant $4L$. On the other hand, there are obvious examples in which some marginals of $X$ need not be subgaussian. For example, if $X$ has independent components that are distributed according to an exponential random variable $y$, then for every $1 \leq i \leq d$, $\|\inr{X,e_i}\|_{\psi_2} = \|y\|_{\psi_2} = \infty$.

\vskip0.3cm
Another simple example are of random vectors with a multivariate $t$-distribution\footnote{See, for example, \cite{Kotz04} for an extensive survey on multivariate $t$-distributions and their properties.}, which, in some cases, satisfy an $L_4-L_2$ norm equivalence but are not $L$-subgaussian for any $L$. The bad subgaussian behaviour is an immediate consequence of the observation that when $d = 1$ and the random variable has $\nu$ degrees of freedom, its $\nu$-th moment does not exist.
\begin{Example}
Assume that $Z$ has a multivariate normal distribution $\mathcal{N}(0, \Sigma^{\prime})$ and $V$ is a random variable independent of $Z$ that has a $\chi^{2}_{\nu}$ distribution for some $\nu \ge 1$. Consider the random vector $X = \frac{Z}{\sqrt{V/\nu}}$, which is centred and has a multivariate $t$-distribution with parameters $(\nu, \Sigma^{\prime})$. Fix $t \in \R^d \setminus \{0\}$ and consider the random variable $\inr{X,t} = \frac{\inr{Z,t}}{\sqrt{V/\nu}}$. Observe that $\inr{Z,t}$ is normal with mean zero and variance $t^T\Sigma^{\prime} t$ and is independent of $V$, and therefore has a $t$ distribution with $\nu$ degrees of freedom. A  straightforward calculation shows that its kurtosis is $\frac{3\nu - 6}{\nu - 4}$ for $\nu > 4$ \cite{Kotz04}. Hence, $X$ satisfies an $L_4-L_2$ norm  equivalence with $L = \left(\frac{3\nu - 6}{\nu - 4}\right)^{\frac{1}{4}}$ provided that $\nu > 4$, but clearly $X$ is not subgaussian.
\end{Example}

\end{document}